\documentclass[11pt]{amsart}

\usepackage{xcolor}
\usepackage{amsmath,amssymb,amsthm,amsfonts,mathtools}
\usepackage{enumerate}
\usepackage{microtype}
\usepackage{needspace}
\usepackage{cite}
\usepackage[hidelinks]{hyperref}
\usepackage{url}
\newtheorem{theorem}{Theorem}[section]
\newtheorem{proposition}[theorem]{Proposition}
\newtheorem{lemma}[theorem]{Lemma}
\newtheorem{corollary}[theorem]{Corollary}

\theoremstyle{definition}
\newtheorem{definition}[theorem]{Definition}

\theoremstyle{remark}
\newtheorem{remark}[theorem]{Remark}

\newcommand{\F}{\mathbb F}

\newcommand{\Fbar}{\overline{\mathbb F}_p}

\title[Maximal curves with few automorphisms]{Maximal Curves with Few Automorphisms}

\author{Herivelto Borges}
\address{Instituto de Matem\'atica e Estat\'istica, Universidade de S\~ao Paulo\\
Rua do Mat\~ao, 1010, 05508-090 S\~ao Paulo, SP, Brazil}
\email{hborges@ime.usp.br}
\thanks{The first author was supported by CNPq (Brazil), grant 310194/2022-9, and by FAPESP (Brazil), grant 2024/00923-6.}

\author{Lucas Reis}
\address{Departamento de Matem\'atica, Universidade Federal de Minas Gerais, UFMG, Belo Horizonte, MG 30270-901, Brazil}
\email{lucasreismat@gmail.com}
\thanks{The second author was supported by CNPq (Brazil), grants 310583/2025-0 and 420721/2025-8, by FAPESP (Brazil), grant 25/21886-4, and by FAPEMIG (Brazil), grant APQ-02427-26.}

\begin{document}

\begin{abstract}

We consider the problem of constructing maximal curves with few automorphisms over finite fields. Our main ingredient is the determination of the \(\operatorname{PGL}_2\!\left(\overline{\mathbb F}_p\right)\)-stabilizer of the roots of the Chebyshev polynomials, apart from two exceptional cases already treated in the literature. We prove that, in all other cases, the only nontrivial M\"obius symmetry is the sign involution. This settles the general case for the automorphism groups of the corresponding Chebyshev curves. We further construct an infinite family of maximal curves of unbounded genus whose only nontrivial automorphism is the hyperelliptic involution. This reveals a contrasting phenomenon to the common heuristic that curves with many rational points tend to have large automorphism groups.

\end{abstract}

\maketitle

\section{Introduction}

Let $\mathcal X$ be a projective, nonsingular, geometrically irreducible
algebraic curve of genus $g$ defined over a finite field $\mathbb F_q$.
The Hasse--Weil bound states that
$$
 \bigl|\#\mathcal X(\mathbb F_q)-(q+1)\bigr|
 \leq 2g\sqrt q,
$$
and a curve attaining the upper bound is called
$\mathbb F_q$-maximal. Maximal curves have been extensively studied,
for their theoretical interest and for their applications to
coding theory and finite geometry; see, for instance,
\cite[Chapter~10]{HKT},\cite{JinMaximalCodes,KorchmarosNagyTimpanella}.
Their study has led to a broad classification program, from the
classical results of Ihara and R\"uck--Stichtenoth to the construction
of maximal curves not covered by the Hermitian curve and subsequent
developments
\cite{Ihara,RuckStichtenoth,GiuliettiKorchmaros,GGS}.
Nevertheless, a general classification remains out of reach, and the
construction of new families is still an active problem.

Another classical problem concerns the size of the full automorphism group
of an algebraic curve. In characteristic zero, the Hurwitz bound gives
$$
 |\operatorname{Aut}(\mathcal X)|\leq 84(g-1),
 \qquad g\geq2.
$$
In positive characteristic, this bound may fail because of wild
ramification. Stichtenoth proved that
$$
 |\operatorname{Aut}(\mathcal X)|<16g^4,
$$
with the Hermitian curve
$\mathcal H_r:\ X^{r+1}+Y^{r+1}+Z^{r+1}=0$, where $r$ is a power
of $p$, as the unique exception \cite{StichtenothAut}. For this curve,
$$
 g(\mathcal H_r)=\frac{r(r-1)}2,
 \qquad
 \operatorname{Aut}(\mathcal H_r)\cong\operatorname{PGU}(3,r).
$$
Henn subsequently showed that only four types of curves can occur when the
order of the automorphism group exceeds $8g^3$ \cite{Henn}. These
results were followed by substantial work on curves with large
automorphism groups in positive characteristic; see, for instance,
\cite{GiuliettiKorchmarosAut,Montanucci}.

Maximal curves and curves with many automorphisms have often appeared
side by side in the literature, the Hermitian curve being the most
familiar example. Many other classical families of such curves also have large automorphism groups, which have played an important role in their construction and classification
\cite{GiuliettiKorchmaros,GGS,GuralnickMalmskogPries}.
Accordingly, the expectation that curves with many rational points tend
to have large automorphism groups has been suggested in the literature
\cite{Fried1999,Weinstein} and is consistent with many of the examples
and constructions that have shaped the subject.

It is natural, however, to ask to what extent maximality can coexist
with a dearth of symmetry. If, as the evidence suggests, such curves are
indeed rare, then, a fortiori, their construction is all the more
compelling. This theoretical question is also of applied interest. In
code-based cryptography, for instance, automorphisms inherited by
algebraic-geometric codes may expose additional structure exploitable in
structural attacks; see, for example,
\cite{FaugereEtAl2016,Barelli}.

In this paper, we show that maximality can also be compatible with  small automorphism groups. We present, to the best of our knowledge, the first explicit family of maximal curves of unbounded genus whose automorphism groups have order bounded by an absolute constant. This suggests a broader study of maximal curves with few automorphisms, including the possible existence of families with trivial automorphism group.

Our hyperelliptic approach naturally leads us to a question concerning
Dickson polynomials of the first kind,
\begin{equation}\label{eq:dickson}
D_n(X,a)=
\sum_{j=0}^{\lfloor n/2\rfloor}
\frac{n}{n-j}\binom{n-j}{j}(-a)^jX^{n-2j}.
\end{equation}
These polynomials have been
studied from several perspectives; see
\cite{WilliamsDickson,AbhyankarCohenZieve,LidlMullenTurnwald} and the
references therein. We contribute to this literature by determining the
projective stabilizer of their root sets. In particular, our analysis
settles the generic case left open in the determination of the
automorphism groups of the Chebyshev curves
\cite{TafazolianTopAut}, for which  a note on  the   historical context can be found in  Section \ref{sec:aut-chebyshev-curves}. We then use this result to present a family of
maximal curves whose full automorphism group is isomorphic to $C_2$.

Here and throughout, we write
$D_n(X)=D_n(X,1)$, and $Z_n$ denotes its set of roots in $\overline{\mathbb F}_p$.
Our main technical result can be stated as follows.

\begin{theorem}\label{thm:stabilizer-Dn}
Let $p>2$, and let $n>4$ be an integer with $p\nmid n$. If there exists a
M\"obius transformation, distinct from $z\mapsto z$ and $z\mapsto -z$,
which permutes $Z_n$, then
\[
2n-1=p^a\qquad\text{or}\qquad 4n-1=p^e
\]
for some positive integers $a,e$.
\end{theorem}

Combining this with a known maximality criterion, we obtain the following family of maximal hyperelliptic curves with the smallest possible  automorphism group.

\begin{theorem}\label{thm:main-hyperelliptic}
Let $q=p^s$, with $p>3$, and let $n>5$ be an odd divisor of $(q+1)/2$. Assume that neither $2n-1$ nor $4n-1$ is a power of $p$. Then the hyperelliptic curve
$$
\mathcal X_n:\qquad y^2=(x+2)D_n(x)
$$
is $\mathbb F_{q^2}$-maximal of genus $g=(n-1)/2$, and its only nontrivial  automorphism is the hyperelliptic involution
$\iota:(x,y)\longmapsto(x,-y)$.
\end{theorem}

The hypotheses of Theorem~\ref{thm:stabilizer-Dn} are satisfied by several explicit  families. We illustrate such families and discuss a related extremal question in Section~\ref{sec:explicit-families}.

The proof of Theorem \ref{thm:stabilizer-Dn} occupies Sections~\ref{sec:auxiliary}--\ref{sec:proof-stabilizer}.
We then apply Theorem~\ref{thm:stabilizer-Dn} to determine the full
automorphism groups of the Chebyshev curves (Theorem \ref{thm:aut-Cn}) and to prove Theorem \ref{thm:main-hyperelliptic}.

\begin{remark}
For each nonzero $a\in \overline{\F}_p$, let $Z_{n}(a)$ denote the set of roots of $D_n(X, a)$. Looking at the explicit expression for $D_n(X, a)$, we see that if $\alpha\in \overline{\F}_p$ satisfies $\alpha^2=a$, then 
$D_n(\alpha X, a)=\alpha^n D_n(X, 1)$. In particular, $Z_{n}(a)=\alpha Z_n$. Therefore, a M\"obius map $\varphi:\mathbb P^1(\overline{\F}_p)\mapsto \mathbb P^1(\overline{\F}_p)$ permutes the set $Z_{n}(a)$ if and only if the M\"obius map $z\mapsto \alpha^{-1} \varphi(\alpha z)$ permutes the set $Z_n$. 
Consequently, Theorem~\ref{thm:stabilizer-Dn} holds for $Z_n(a)$ for every $a\in\overline{\F}_p^*$.
\end{remark}

\section{Auxiliary results}\label{sec:auxiliary}
Throughout this paper, $p$ denotes an odd prime number. In this section, we provide relevant machinery along with some preliminary results that will be used later on.  We recall some standard properties of Dickson polynomials; see \cite[Chapter~2]{LidlMullenTurnwald}.

\begin{lemma}\label{lem:dick-properties}
Let $m,n$ be positive integers. Then the following hold.
\begin{enumerate}[\upshape 1.]
    \item $D_{mn}(X)=D_m(D_n(X))$,
    \item if $p\nmid n$, then $D_n(X)$ is separable in $\F_p[x]$.
\end{enumerate}
Moreover, we have the following identities:
\begin{equation}\label{eq:x+1/x}
D_n\left(X+X^{-1}\right)=X^n+X^{-n},
\end{equation}
and
\begin{equation}\label{eq:dif-dickson}
(X^2-4)D_n''(X)+XD_n'(X)-n^2D_n(X)=0.
\end{equation}
\end{lemma}

\subsection{The M\"obius action and affine symmetries}

Given
\[
A=\begin{pmatrix}a&b\\ c&d\end{pmatrix}\in \mathrm{GL}_2(\overline{\F}_p)
\]
and $f\in\overline{\F}_p[X]$, we set
\[
A\circ f=(cX+d)^{\deg f}f\left(\frac{aX+b}{cX+d}\right).
\]
For each $\alpha\in\overline{\F}_p$ with $\alpha\ne-d/c$ if $c\ne0$, we set
\[
A\circ\alpha=\frac{a\alpha+b}{c\alpha+d}.
\]

\begin{definition}
A polynomial $f\in\overline{\F}_p[X]$ is \emph{invariant} by $A$ if $A\circ f=\lambda f$ for some $\lambda\in\overline{\F}_p$.
\end{definition}

For $A,B\in\mathrm{GL}_2(\overline{\F}_p)$, we write $A\sim B$ if there exists $\delta\in\overline{\F}_p^*$ such that $A=\delta B$. The following lemma, whose proof is straightforward, will be frequently used without further mention.

\begin{lemma}\label{lem:action}
Let $f\in\overline{\F}_p[X]$ be a nonconstant polynomial, and fix $A,B\in\mathrm{GL}_2(\overline{\F}_p)$.
\begin{enumerate}[\upshape (a)]
\item If $f$ is separable, then $f$ is invariant by $A$ if and only if the map $z\mapsto A\circ z$ permutes the roots of $f$. In the affirmative case, $f$ is also invariant by $A^{-1}$
\item If $f$ is invariant by $A$ and $B$, then $f$ is invariant by $AB$
\item If $f$ is invariant by $A$, then $f$ is invariant by any $C\in\mathrm{GL}_2(\overline{\F}_p)$ with $A\sim C$.
\end{enumerate}
\end{lemma}

In the next lemma, we  determine the upper triangular matrices under which $D_n(X)$ is invariant.

\begin{lemma}\label{lem:affine}
If $n\ge 2$ is not divisible by $p$, and $D_n(X)$ is invariant by
\[
\begin{pmatrix}a&b\\0&1\end{pmatrix},
\]
then $b=0$ and $a=\pm1$.
\end{lemma}

\begin{proof}
From the hypothesis, $D_n(aX+b)=tD_n(X)$ for some nonzero $t\in\overline{\F}_p$. Comparing the coefficients of $X^{n-1}$ and using Eq.~\eqref{eq:dickson} yield $na^{n-1}b=0$. Since $a\ne0$ and $p\nmid n$, we conclude that $b=0$. Thus $D_n(aX)=tD_n(X)$. Comparing the coefficients of $X^n$ and $X^{n-2}$ gives $t=a^n$ and $-tn=-a^{n-2}n$. Hence $a^n=t=a^{n-2}$, so $a^2=1$ and therefore $a=\pm1$.
\end{proof}

\subsection{$D_n$ and its companion polynomial}




For each positive integer $n$, set
\[
B_n(X)=\sum_{j=0}^n\binom{2n}{2j}X^j.
\]
In particular,
\begin{equation}\label{eq:bn-id}
B_n(X^2)=\frac{(1+X)^{2n}+(1-X)^{2n}}{2}.    
\end{equation}

\begin{remark}\label{rem:Bn}
Since $\binom{2n}{2j}=\binom{2n}{2(n-j)}$, $B_n(X)$ is invariant by
\[
\begin{pmatrix}0&1\\1&0\end{pmatrix}.
\]
\end{remark}

For the rest of the paper, we set
\begin{equation}\label{tau-def}
   \sigma=\begin{pmatrix}-1&0\\0&1\end{pmatrix},
\qquad
\tau=\begin{pmatrix}1&-2\\1&2\end{pmatrix}.
\end{equation}

Let $Z_n$ and $W_n$ denote the sets of roots of $D_n$ and $B_n$, respectively.

\begin{lemma}\label{lem:Bn-1}
For every $n\ge1$, we have
\begin{equation}\label{Dn-Bn}
4^nD_n(X)=2(X+2)^nB_n(\tau(X)).
\end{equation}
Consequently,
\[
W_n=\{\tau\circ\alpha:\alpha\in Z_n\}.
\]
In particular, if $p\nmid n$ and $D_n$ is invariant by $A$, then $B_n$ is invariant by $\tau A\tau^{-1}$.
\end{lemma}

\begin{proof}
For the first statement, it suffices to check Eq.~\eqref{Dn-Bn} for $X=Y^2+Y^{-2}$. Eq.~\eqref{eq:x+1/x} yields 
$$D_n(Y^2+Y^{-2})=Y^{2n}+Y^{-2n}.$$
On the other hand, $\tau(Y^2+Y^{-2})=\left(\frac{Y-Y^{-1}}{Y+Y^{-1}}\right)^2$ and so Eq.~\eqref{eq:bn-id} yields
$$B_n(\tau(Y^2+Y^{-2}))=2^{-1}(Y+Y^{-1})^{-2n}((2Y)^{2n}+(2Y^{-1})^{2n}).$$
Since $(Y^2+Y^{-2}+2)^n=(Y+Y^{-1})^{2n}$, the result follows. The remaining statements follow by Eq.~\eqref{Dn-Bn} and Lemma~\ref{lem:action}.

\end{proof}



Next we classify the integers $n$ for which $B_n(X)$ can be invariant by a nontrivial diagonal element.

\begin{lemma}\label{lem:Bn-2}
If $2n\equiv1\pmod p$, and $B_n(X)$ is invariant by
\[
\begin{pmatrix}\lambda&0\\0&1\end{pmatrix}
\]
for some $\lambda\ne1$, then there exists $r\ge1$ such that $2n=p^r+1$.
\end{lemma}

\begin{proof}
By hypothesis, there exists $\delta\in\overline{\F}_p^*$ such that $B_n(\lambda X)=\delta B_n(X)$. Comparing constant terms gives $\delta=1$. Thus $B_n(\lambda X)=B_n(X)$. Let $\ell\in\overline{\F}_p$ satisfy $\ell^2=\lambda$. Since $\lambda\ne1$, we have $\ell\ne\pm1$. Using
\[
2B_n(X^2)=(1+X)^{2n}+(1-X)^{2n},
\]
we obtain
\[
(1+\ell X)^{2n}+(1-\ell X)^{2n}
=(1+X)^{2n}+(1-X)^{2n}.
\]
Taking derivatives and using $2n\equiv1\pmod p$ gives
\[
\ell(1+\ell X)^{2n-1}-\ell(1-\ell X)^{2n-1}
=(1+X)^{2n-1}-(1-X)^{2n-1}.
\]
Eliminating the powers of $1-X$ from these two identities yields
\begin{equation}\label{eq:frob}
(1+\ell)(1+\ell X)^{2n-1}
+(1-\ell)(1-\ell X)^{2n-1}
=2(1+X)^{2n-1}.
\end{equation}
Write $2n-1=p^r M$, where  $M\ge1$ is not divisible by $p$. Setting $Y=X^{p^r}$, Eq.~\eqref{eq:frob} gives
\[
\alpha(1+tY)^M+\beta(1-tY)^M=(1+Y)^M
\]
for some nonzero $\alpha,\beta,t\in\overline{\F}_p$ with $t\ne\pm1$. If $M>1$, then $Y=-1$ is a multiple root of the right-hand side. Evaluating the identity and its derivative at $Y=-1$ gives
\[
\begin{cases}
\alpha(1-t)^M+\beta(1+t)^M=0,\\
M\alpha t(1-t)^{M-1}-M\beta t(1+t)^{M-1}=0.
\end{cases}
\]
Since $p\nmid M$, and $\alpha,\beta,t,1\pm t$ are nonzero, we obtain
\[
(1-t)^M(1+t)^{M-1}
=-(1-t)^{M-1}(1+t)^M,
\]
and hence $1-t=-(1+t)$, which gives $2=0$, a contradiction. Therefore, $M=1$, thus $2n-1=p^r$.
\end{proof}

We first record a general result that will be useful in the proofs of Lemma~\ref{lem:Bn-infinity} and Proposition~\ref{prop:main}.

\begin{proposition}\label{prop:differential-reduction}

Let $K$ be a field and  $f\in K[X]$ be a separable polynomial satisfying
\begin{equation}\label{eq:general-differential}
P_2(X)f''(X)+P_1(X)f'(X)+\rho f(X)=0,
\end{equation}
where $\deg P_2\le2$, $\deg P_1\le1$, and $\rho\in K$.

Let
\[
\varphi(X)=\frac{aX+b}{cX+d},
\qquad
H=cX+d,
\qquad
\Delta=ad-bc\neq 0,
\]
and suppose that
\begin{equation}\label{eq:general-mobius}
H^{m}f(\varphi(X))=\lambda f(X)
\end{equation}
for some $\lambda\in K^*$ and some positive integer $m$.

Set
\[
\widetilde P_2=H^2P_2(\varphi),
\qquad
\widetilde P_1=HP_1(\varphi).
\]

Then
\[
H\Psi(X)f'(X)+C(X)f(X)=0
\]
for some $C(X)\in K[X]$, where
\begin{equation}\label{eq:Psi-general}
\Psi
=
P_2\bigl(2c(1-m)\widetilde P_2+\Delta\widetilde P_1\bigr)
-H\widetilde P_2P_1.
\end{equation}

Moreover, $\deg\Psi\le4$ and
\[
\frac{f}{\gcd(f,H)}\mid\Psi.
\]

In particular, if
\[
\deg f-\deg\gcd(f,H)>4,
\]
then $\Psi=0$.

\end{proposition}

\begin{proof}

Set $g=f\circ\varphi$. Since $\varphi'=\Delta H^{-2}$ and
$\varphi''=-2c\Delta H^{-3}$, differentiating $g=f\circ\varphi$ gives
\[
f'(\varphi)=\frac{H^2}{\Delta}g',
\qquad
f''(\varphi)
=
\frac{H^4}{\Delta^2}g''
+\frac{2cH^3}{\Delta^2}g'.
\]

On the other hand, differentiating $g=\lambda H^{-m}f$ gives
\[
\begin{aligned}
g'&=\lambda H^{-m}
\left(f'-\frac{mc}{H}f\right),\\
g''&=\lambda H^{-m}
\left(
f''-\frac{2mc}{H}f'
+\frac{m(m+1)c^2}{H^2}f
\right).
\end{aligned}
\]

Substituting these expressions into
\eqref{eq:general-differential} evaluated at $\varphi(X)$ gives
\[
H^2\widetilde P_2f''
+
H\bigl(2c(1-m)\widetilde P_2+\Delta\widetilde P_1\bigr)f'
+
C_0f=0
\]
for some $C_0\in K[X]$.
Eliminating $f''$ by means of \eqref{eq:general-differential} yields
\[
H\Psi f'+Cf=0,
\]
where $\Psi$ is given by \eqref{eq:Psi-general}.
Since $\deg P_2,\deg\widetilde P_2\le2$ and
$\deg P_1,\deg\widetilde P_1\le1$, we have $\deg\Psi\le4$.
Since $f$ is separable, $\gcd(f,f')=1$, and hence $f\mid H\Psi$.
Therefore,
\[
\frac{f}{\gcd(f,H)}\mid\Psi.
\]
The last assertion follows from the degree bound.

\end{proof}

We establish the following additional property of $W_n$ for later use.

\begin{lemma}\label{lem:Bn-infinity}

Let $p>3$ and let $n>5$ be an integer not divisible by $p$.
If a M\"obius transformation permutes $W_n\cup\{\infty\}$, then it
fixes $\infty$.

\end{lemma}

\begin{proof}

Suppose that a M\"obius transformation $\varphi$ permutes
$W_n\cup\{\infty\}$ and does not fix $\infty$.
Then there exist $u,v\in W_n$ such that
$\varphi(u)=\infty$ and $\varphi(\infty)=v$.
Hence
\[
\varphi(X)=v+\frac{\kappa}{X-u}
\]
for some $\kappa\in\Fbar^*$.

Since $p\nmid n$, the polynomial $B_n$ is separable.
As $\varphi$ maps $W_n\setminus\{u\}$ onto $W_n\setminus\{v\}$ and
$v$ is a simple root of $B_n$, the polynomial
$(X-u)^nB_n(\varphi(X))$ has degree $n-1$, and its roots are precisely
the elements of $W_n\setminus\{u\}$.

Therefore, for some $\lambda\in\Fbar^*$,
\begin{equation}\label{eq:Bn-root-infinity}
(X-u)^{n+1}
B_n\left(v+\frac{\kappa}{X-u}\right)
=
\lambda B_n(X).
\end{equation}

Using Eq.~\eqref{eq:dif-dickson} and Eq.~\eqref{Dn-Bn}, one checks that
$f=B_n$ satisfies
\begin{equation}\label{eq:dif-Bn}
2X(1-X)f''
+
\bigl(1+(4n-3)X\bigr)f'
-n(2n-1)f
=0.
\end{equation}

We apply Proposition~\ref{prop:differential-reduction} with
\[
P_2=2X(1-X),
\qquad
P_1=1+(4n-3)X,
\qquad
\rho=-n(2n-1),
\]
and
\[
H=X-u,
\qquad
\Delta=-\kappa,
\qquad
m=n+1.
\]

Since $u$ is a simple root of $f$, we have $\gcd(f,H)=H$.
As $n>5$,
\[
\deg f-\deg H=n-1>4,
\]
and Proposition~\ref{prop:differential-reduction} gives $\Psi=0$.
For the above choices, a direct expansion of
Eq.~\eqref{eq:Psi-general} shows that the coefficient of $X^4$
in $\Psi$ is
$6v(1-v)$. Thus $6v(1-v)=0$.

Finally, $B_n(0)=1$ and $B_n(1)=2^{2n-1}\ne0$, so
$0,1\notin W_n$. Since $v\in W_n$, we have $v(1-v)\ne0$,
contradicting $p>3$.
Therefore, $\varphi$ fixes $\infty$.

\end{proof}

\section{A fundamental reduction}\label{sec:fundamental}
We observe from Eq.~\eqref{eq:dickson} that $D_n(X)$ is invariant by $\sigma:z\mapsto-z$ for every $n\ge1$. In the following lemma, we show that if $D_n$ is invariant by a M\"obius map distinct from $z\mapsto\pm z$, then it must be invariant by a special M\"obius map.

\begin{lemma}\label{lem:Mab}
Let $n\ge 2$ be an integer not divisible by $p$. If $D_n$ is invariant by some $A\not\sim\sigma,\mathrm{id}$, then there exist $a,b\in\overline{\F}_p$ such that $D_n$ is invariant by
\[
M_{a,b}:=\begin{pmatrix}a&b\\1&-a\end{pmatrix}.
\]
\end{lemma}

\begin{proof}
Write
\[
A=\begin{pmatrix}x&y\\z&w\end{pmatrix}.
\]
Since $A\not\sim\sigma,\mathrm{id}$, Lemma~\ref{lem:affine} implies $z\ne0$. In particular, $D_n$ is also invariant by
\[
A\sigma A^{-1}\sim
\begin{pmatrix}
wx+yz&-2xy\\
2wz&-(wx+yz)
\end{pmatrix}
\not\sim\mathrm{id}.
\]
If $w\ne0$, then $2wz\ne0$, and hence $A\sigma A^{-1}\sim M_{a,b}$ for some $a,b\in\overline{\F}_p$. If $w=0$, Lemma~\ref{lem:affine} implies that $-2xy=0$ or $A\sigma A^{-1}\sim\sigma$. In both cases, $-2xy=0$. Since $w=0$, we have $y\ne0$, and thus $x=0$. Therefore,  $A\sim M_{0,b}$ for some nonzero $b\in\overline{\F}_p$.
\end{proof}

Next we study the possible values of $n,a,b$ for which $D_n(X)$ is invariant by $M_{a,b}$.

\begin{proposition}\label{prop:main}

Fix $n>4$ not divisible by $p$ and assume that $D_n$ is invariant by
$M_{a,b}$. If $p\ge3$ and $a^2-4\ne0$, then
\[
2n\equiv1\pmod p
\quad\text{and}\quad
(a,b)\in\{(a,-4),(0,\pm4)\}.
\]
\end{proposition}

\begin{proof}
Set
\[
s:=\det M_{a,b}=-a^2-b,
\qquad
u=aX+b,
\qquad
v=X-a,
\qquad
R=u^2-4v^2.
\]
Since $D_n$ is invariant by $M_{a,b}$, there exists $t\in\Fbar^*$ such
that
\[
v^{n}D_n\left(\frac{u}{v}\right)=tD_n(X).
\]
We apply  Proposition~\ref{prop:differential-reduction} with
\[
f=D_n,
\qquad
P_2=X^2-4,
\qquad
P_1=X,
\qquad
\rho=-n^2,
\]
and
\[
H=v,
\qquad
\Delta=s,
\qquad
m=n.
\]
In this case, $\widetilde P_2=R$ and $\widetilde P_1=u,$ and hence
\begin{equation}\label{eq:Psi-Dn}
\Psi(X)
=
(X^2-4)\bigl(2(1-n)R+su\bigr)-XvR.
\end{equation}

Since $p\nmid n$, the polynomial $D_n$ is separable.
Moreover, as $M_{a,b}$ permutes its roots, no root of $D_n$ is a zero of
$v$. Thus $\gcd(D_n,v)=1$.

Since $n>4$, Proposition~\ref{prop:differential-reduction} gives
$\Psi=0$.
Expanding \eqref{eq:Psi-Dn}, the coefficients of $X^4,  X^3$ and $X^2$ give
$$
\left\{
\begin{array}{l}
(1-2n)(a^2-4)=0,\\
a(4n-1)(a^2+s-4)=0,\\
-2n(a^2+s)^2-a^2(a^2+s)+4(4n-1)a^2-32(n-1)=0.
\end{array}
\right.
$$
Since $a^2-4\ne0$, the first equation gives $2n-1\equiv0\pmod p$. In particular, $4n-1\equiv1\pmod p$, and the second equation implies that either $a=0$ or $a^2+s-4=0$.
If $a\ne0$, then $s=-a^2-b$ gives $b=-4$. If $a=0$, the third equation gives
$$
-2ns^2-32(n-1)=0.
$$
Using $2n-1\equiv0\pmod p$, we obtain $b^2=s^2=16$, and hence $b=\pm4$.

\end{proof}

\section{Proof of Theorem~\ref{thm:stabilizer-Dn}}\label{sec:proof-stabilizer}

\begin{proof}[Proof of Theorem~\ref{thm:stabilizer-Dn}]
Let $n>4$ be not divisible by $p$, and suppose that $D_n$ is invariant
under some M\"obius map distinct from $\sigma$ and $\mathrm{id}$.
By Lemma~\ref{lem:Mab}, $D_n$ is invariant under $M_{a,b}$ for some
$a,b$. Proposition~\ref{prop:main} implies that either $a^2=4$, or
$2n\equiv1\pmod p$ and
$(a,b)\in\{(a,-4),(0,\pm4)\}$.
Since $D_n$ is invariant under $\sigma$ and
$\sigma M_{0,b}\sim M_{0,-b}$, we may assume $b=-4$ whenever
$a^2\ne4$. Thus there are two cases.

\begin{enumerate}[(i)]
\item Suppose $b=-4$. By Lemma~\ref{lem:Bn-1}, $B_n(X)$ is invariant
under
\[
\begin{pmatrix}1&-2\\1&2\end{pmatrix}
\begin{pmatrix}a&-4\\1&-a\end{pmatrix}
\begin{pmatrix}1&-2\\1&2\end{pmatrix}^{-1}
\sim
\begin{pmatrix}0&a-2\\a+2&0\end{pmatrix}.
\]
By Remark~\ref{rem:Bn}, it is also invariant under the inversion
$X\mapsto X^{-1}$, and hence under the scaling
$X\mapsto\lambda X$, where
$\lambda=(a-2)/(a+2)$. Since $a^2\ne4$, we have
$\lambda\ne0,1$. As $2n\equiv1\pmod p$, Lemma~\ref{lem:Bn-2}
yields $2n=p^r+1$ for some $r\ge1$.

\item Suppose $a^2=4$. Since
$\sigma M_{a,b}\sigma^{-1}\sim M_{-a,b}$, we may assume $a=-2$.
Set $\Delta=a^2+b=b+4$. Then
\[
(X+2)^nD_n\left(\frac{\Delta}{X+2}-2\right)=tD_n(X).
\]
Replacing $X$ by $X^2-2$, choose
$\delta\in\overline{\F}_p$ with $\delta^2=\Delta$. Since
$D_2(X)=X^2-2$ and
$D_n(D_2(X))=D_{2n}(X)$ by Lemma~\ref{lem:dick-properties}, we obtain
\[
X^{2n}D_{2n}\left(\frac{\delta}{X}\right)=tD_{2n}(X).
\]
Thus $D_{2n}$ is invariant under $M_{0,\delta}$.
Proposition~\ref{prop:main}, applied to $D_{2n}$, gives
$4n\equiv1\pmod p$ and $\delta=\pm4$. As before, we may assume
$\delta=-4$. Applying part~(i) to $D_{2n}$ and using
Lemma~\ref{lem:Bn-2}, we obtain $4n=p^r+1$ for some $r\ge1$.
The result follows.
\end{enumerate}
\end{proof}

For later use, we record the following consequence of
Theorem~\ref{thm:stabilizer-Dn}, concerning the stabilizer of the
enlarged set $Z_n\cup\{-2\}$.

\begin{corollary}\label{cor:stabilizer-extended}
Let $p>3$, and let $n>5$ be such that $p\nmid n$. Set
\[
Z_n=\{\alpha\in\overline{\mathbb F}_p:D_n(\alpha)=0\},
\qquad
\widetilde Z_n=Z_n\cup\{-2\}.
\]
If neither $2n-1$ nor $4n-1$ is a power of $p$, then the identity is
the only M\"obius transformation that permutes $\widetilde Z_n$.
\end{corollary}

\begin{proof}
Let $\theta$ be a nontrivial M\"obius transformation that permutes
$\widetilde Z_n$.
Suppose  that $\theta(-2)=-2$. Then $\theta$ permutes $Z_n$, and
Theorem~\ref{thm:stabilizer-Dn} gives
$\theta(X)=-X,$
contradicting $\theta(-2)=-2$.

Suppose now that
$\theta(-2)=\beta\in Z_n.$
Since $\theta$ permutes $\widetilde Z_n$, there exists $\alpha\in Z_n$
such that $\theta(\alpha)=-2.$
Set
\[
\gamma=\tau\theta\tau^{-1},
\qquad
u=\tau(\alpha),
\qquad
v=\tau(\beta),
\]
 where $\tau$ is as in \eqref{tau-def}.
By Lemma~\ref{lem:Bn-1}, $\tau$ maps $Z_n$ bijectively onto $W_n$, while
$\tau(-2)=\infty.$
Hence $\gamma$ permutes
$W_n\cup\{\infty\}.
$
Moreover,
\[
\gamma(u)=\infty,
\qquad
\gamma(\infty)=v\in W_n.
\]
Thus $\gamma$ does not fix $\infty$, contradicting
Lemma~\ref{lem:Bn-infinity}.
Therefore, no nontrivial M\"obius transformation permutes
$\widetilde Z_n$.
\end{proof}

\section{Special curves and their automorphisms}
\label{sec:aut-chebyshev-curves}

\subsection{The Chebyshev curves}
Let $p>2$, and let $n>4$ be such that
$p\nmid n$. Consider the smooth plane curve
\[
\mathcal C_n:\qquad y^n=D_n(x).
\]
Theorem~\ref{thm:stabilizer-Dn} completes the general case left open in
\cite[Expectation~3.19]{TafazolianTopAut} and, together with the
exceptional cases determined therein, yields the following classification\protect\footnotemark[1].
\footnotetext[1]{\textbf{Historical Note}:
During the \emph{Workshop on Algebraic Curves and Function Fields over a
Finite Field}, held in Perugia in February 2015, a question concerning the
existence of nonsingular plane maximal curves beyond the Fermat and Hurwitz
families was posed, with a symbolic prize announced for a solution. A
positive answer was presented during the same meeting through the family
$\mathcal{C}_n:\,y^n=D_n(x)$, here called Chebyshev curves. These curves
provided a new family beyond the two known families. This result was later
formalized and published in \cite{ArakelianBorgesSpeziali}, although the full
automorphism groups were not determined there. The two exceptional cases in
which either $2n-1$ or $4n-1$ is a power of $p$ were subsequently settled in
\cite{TafazolianTopAut}. The present work completes the picture by treating
the general case.}

\begin{theorem}\label{thm:aut-Cn}
Let $p>2$, and let $n>4$, with
$p\nmid n$. Then
\[
\operatorname{Aut}_{\overline{\mathbb F}_p}(\mathcal C_n)\simeq
\begin{cases}
(C_n\times C_n)\rtimes S_3,
    & 2n-1=p^a,\\[1mm]
C_n\times S_3,
    & 4n-1=p^e,\\[1mm]
C_n\times C_2,
    & \text{otherwise},
\end{cases}
\]
where $a,e$ are positive integers.
\end{theorem}

\begin{proof}
If $2n-1=p^a$, the result follows from
\cite[Propositions~3.1 and~3.3]{TafazolianTopAut}; if
$4n-1=p^e$,  from Theorem~3.17 together with
Proposition~3.15 of \cite{TafazolianTopAut}. Thus it remains to consider the
case in which neither $2n-1$ nor $4n-1$ is a power of $p$.

Let $Z_n=\{\alpha\in \overline{\mathbb F}_p:D_n(\alpha)=0\}$ and
\[
N=\{(x,y)\mapsto(x,\zeta y):\zeta^n=1\}\simeq C_n.
\]
Since $2n-1$ is not a power of $p$, Proposition~2.1 and Lemma~3.4 of
\cite{TafazolianTopAut} give
\[
\operatorname{Aut}_{\overline{\mathbb F}_p}(\mathcal C_n)/N
\simeq
\operatorname{Stab}_{\operatorname{PGL}_2(\overline{\mathbb F}_p)}(Z_n).
\]
By Theorem~\ref{thm:stabilizer-Dn}, the assumption that neither
$2n-1$ nor $4n-1$ is a power of $p$ yields
\[
\operatorname{Stab}_{\operatorname{PGL}_2(\overline{\mathbb F}_p)}(Z_n)
=
\langle x\mapsto-x\rangle\simeq C_2.
\]
Since the involution $\rho:(x,y)\mapsto(-x,-y)$ is an automorphism of
$\mathcal C_n$, with $\rho\notin N$, and commutes with $N$, we have
$\langle N,\rho\rangle\simeq C_n\times C_2$. Since
$\operatorname{Aut}_{\overline{\mathbb F}_p}(\mathcal C_n)/N\simeq C_2$,
it follows that
\[
\operatorname{Aut}_{\overline{\mathbb F}_p}(\mathcal C_n)
=
\langle N,\rho\rangle
\simeq C_n\times C_2.
\]
\end{proof}

\subsection{Maximal hyperelliptic curves with few automorphisms}
\label{sec:maximal-hyperelliptic}

The following family is not our main target, but it provides another natural class of maximal curves with few automorphisms. Since, for even \(n\), it arises naturally as a subcover of the Chebyshev curve considered above, we record this additional case here.

\begin{theorem}\label{thm:maximal-hyperelliptic-Dn}
Let $q=p^s$, with $p>2$, and let $n>4$ be an even divisor of
$(q+1)/2$. If neither $2n-1$ nor $4n-1$ is a power of $p$, then
the curve
\[
\mathcal X_n:\qquad y^2=D_n(x)
\]
is $\mathbb F_{q^2}$-maximal of genus $g=(n-2)/2$, and
\[
\operatorname{Aut}_{\overline{\mathbb F}_p}(\mathcal X_n)
=
\{(x,y)\mapsto(\pm x,\pm y)\}
\simeq C_2\times C_2.
\]
\end{theorem}

\begin{proof}
Since $n\mid(q+1)/2$, we have $p\nmid n$, and hence $D_n$ is separable.
Since $n$ is even, $\deg D_n=n=2g+2$, where $g=(n-2)/2$.

Consider the curve
\[
\mathcal C_n:\qquad v^n=D_n(u).
\]
Under the assumption $n\mid(q+1)/2$, the curve $\mathcal C_n$ is
$\mathbb F_{q^2}$-maximal; see
\cite[Theorem~6.3]{ArakelianBorgesSpeziali} and
\cite{KazemifardTafazolianTorres}. Since $n$ is even,
$(u,v)\mapsto(u,v^{n/2})$ defines an $\mathbb F_q$-covering
$\mathcal C_n\to\mathcal X_n$. Hence $\mathcal X_n$ is
$\mathbb F_{q^2}$-maximal by the Kleiman--Serre covering result
\cite[Proposition~6]{Lachaud}.

Let $\iota:(x,y)\mapsto(x,-y)$ be the hyperelliptic involution.
The branch locus of the hyperelliptic covering is $Z_n$, and therefore,
\[
\operatorname{Aut}_{\overline{\mathbb F}_p}(\mathcal X_n)/
\langle\iota\rangle
\simeq
\operatorname{Stab}_{\operatorname{PGL}_2(\overline{\mathbb F}_p)}(Z_n).
\]
By Theorem~\ref{thm:stabilizer-Dn},
\[
\operatorname{Stab}_{\operatorname{PGL}_2(\overline{\mathbb F}_p)}(Z_n)
=
\langle x\mapsto-x\rangle
\simeq C_2.
\]
It is clear that $ x\mapsto-x$ lifts to the involution
$\theta:(x,y)\mapsto(-x,y)$, 
with $\theta \notin \langle\iota \rangle$. Thus
\[
\operatorname{Aut}_{\overline{\mathbb F}_p}(\mathcal X_n)
=
\langle\iota,\theta\rangle
=
\{(x,y)\mapsto(\pm x,\pm y)\}
\simeq C_2\times C_2.
\]
\end{proof}

We now turn to our main target, constructed from Chebyshev polynomials of odd degree. In this case, the branch locus is obtained by adjoining \(-2\) to the roots of \(D_n\), and Corollary~\ref{cor:stabilizer-extended} applies.

\begin{proof}[Proof of Theorem~\ref{thm:main-hyperelliptic}]
Since $n\mid(q+1)/2$, we have $q\equiv-1\pmod{2n}$.
In the notation of \cite{TafazolianTopMaximal},
$\varphi_n(X)=D_n(X)$. Since $n$ is odd,
\cite[Theorem~1.4]{TafazolianTopMaximal} implies that
\[
\mathcal X_n:\qquad y^2=(x+2)D_n(x)
\]
is an $\F_{q^2}$-maximal curve of genus $g=(n-1)/2$.

Since $p\nmid n$, the polynomial $D_n$ is separable, and
$D_n(-2)=2(-1)^n\ne0$. Hence the branch locus is
$\widetilde Z_n=Z_n\cup\{-2\}.
$
Let $\iota:(x,y)\mapsto(x,-y)$ be the hyperelliptic involution. Then
\[
\operatorname{Aut}_{\overline{\mathbb F}_p}(\mathcal X_n)/
\langle\iota\rangle
\simeq
\operatorname{Stab}_{\operatorname{PGL}_2(\overline{\mathbb F}_p)}
(\widetilde Z_n).
\]
By Corollary~\ref{cor:stabilizer-extended}, this stabilizer  is
trivial. Hence
$\operatorname{Aut}_{\overline{\mathbb F}_p}(\mathcal X_n)
=
\langle\iota\rangle
\simeq C_2.$
\end{proof}

\section{Explicit Families and an Extremal Question}\label{sec:explicit-families}

\subsection*{Explicit Families}

To illustrate the scope of Theorems~\ref{thm:main-hyperelliptic} and \ref{thm:maximal-hyperelliptic-Dn}, we now single out explicit infinite families satisfying their hypotheses.

\subsubsection*{Case \(n\) odd}

Fix \(p>3\) and \(d \geq 1\), and for each \(r \geq 1\) set
\[
q_{r,d}=p^{d(2r+1)}, \qquad
n_{r,d}=\frac{p^{d(2r+1)}+1}{p^d+1}.
\]
Since
\[
n_{r,d}
=
1+\left(p^d-1\right)
\left(p^d+p^{3d}+\cdots+p^{(2r-1)d}\right),
\]
\(n_{r,d}\) is odd. Moreover, \(q_{r,d}+1=\left(p^d+1\right)n_{r,d}\), so \(n_{r,d}\mid \left(q_{r,d}+1\right)/2\). Finally, \(n_{r,d}\equiv 1 \pmod p\), and hence
\[
2n_{r,d}-1\equiv 1 \pmod p,
\qquad
4n_{r,d}-1\equiv 3 \pmod p.
\]
Thus neither \(2n_{r,d}-1\) nor \(4n_{r,d}-1\) is a power of \(p\), and Theorem~\ref{thm:main-hyperelliptic} applies to
\[
\mathcal{X}_{r,d}:\qquad
y^2=(x+2)D_{n_{r,d}}(x).
\]
Therefore, \(\mathcal{X}_{r,d}\) is \(\mathbb{F}_{q_{r,d}^2}\)-maximal of genus
\begin{equation}
g\left(\mathcal{X}_{r,d}\right)
=
\frac{n_{r,d}-1}{2}
=
\frac{p^d\left(p^{2rd}-1\right)}
{2\left(p^d+1\right)},
\label{eq:genus-odd-family}
\end{equation}
and, for fixed \(p\) and \(d\), we have \(g\left(\mathcal{X}_{r,d}\right)\to\infty\) while \(\operatorname{Aut}_{\overline{\mathbb F}_p}\left(\mathcal{X}_{r,d}\right)\simeq C_2\).

\subsubsection*{Case \(n\) even}

Fix \(p\equiv3\pmod4\) and an odd integer \(d\geq1\), with \((p,d)\neq(3,1),(7,1)\), and for each \(r\geq1\) set
\[
q_{r,d}=p^{d(2r+1)},\qquad
n_{r,d}=2\frac{p^{d(2r+1)}+1}{p^d+1}.
\]
Then \(n_{r,d}\) is even, and since \(p^d\equiv3\pmod4\),
$
\frac{q_{r,d}+1}{2}
=
\frac{p^d+1}{4}\,n_{r,d},
$
so \(n_{r,d}\mid(q_{r,d}+1)/2\). One checks that, since \((p,d)\neq(3,1),(7,1)\), neither \(2n_{r,d}-1\) nor \(4n_{r,d}-1\) is a power of \(p\). Therefore Theorem~\ref{thm:maximal-hyperelliptic-Dn} applies to
\[
\mathcal Y_{r,d}:\qquad y^2=D_{n_{r,d}}(x).
\]
Hence \(\mathcal Y_{r,d}\) is \(\mathbb F_{q_{r,d}^2}\)-maximal of genus
\[
g(\mathcal Y_{r,d})
=
\frac{n_{r,d}-2}{2}
=
\frac{p^d\left(p^{2rd}-1\right)}{p^d+1},
\]
and, for fixed \(p\) and \(d\), we have \(g(\mathcal Y_{r,d})\to\infty\) while \(\operatorname{Aut}_{\overline{\mathbb F}_p}\left(\mathcal Y_{r,d}\right)\simeq C_2\times C_2\).

\subsection*{An extremal question}

The preceding results suggest a natural extremal question. 
Let \(\mathcal X/\mathbb F_{q^2}\) be a maximal hyperelliptic curve of genus \(g\), with \(q\) odd. It follows from \cite[Corollary~3.6 and Theorem~6.1]{GarciaTafazolian} that
\[
g\leq \frac{q-1}{2},
\]
with equality if and only if \(\mathcal X\) is isomorphic to the Roquette curve
\[
y^2=x^q+x.
\]
The automorphism group of this curve is exceptionally large; indeed, it is one of the four types of curves occurring in Henn's classification of curves with
\(\lvert\operatorname{Aut}(\mathcal X)\rvert\geq 8g^3\)
\cite{Henn}. Since every hyperelliptic curve has at least the hyperelliptic involution, it is natural to define
\[
\mu(q)=\max\left\{g(\mathcal X):
\mathcal X/\mathbb F_{q^2}\text{ is maximal hyperelliptic and }
\operatorname{Aut}_{\overline{\mathbb F}_p}(\mathcal X)\simeq C_2
\right\},
\]
whenever such curves exist; otherwise, we set \(\mu(q)=-1\). For $q$ odd, the  discussion yields
\[
\mu(q)\leq \frac{q-3}{2}.
\]
Note that our construction gives lower bounds for certain values of $q$. Indeed, considering the  curves $\mathcal{X}_{r,d}$  above, we have
\[
\mu(q_{r,d})\geq \frac{q_{r,d}-p^d}{2(p^d+1)}.
\]

For some small primes \(p\), the value of \(\mu(p)\) can be determined by combining explicit results on superspecial hyperelliptic curves in small genus; see Ibukiyama--Katsura--Oort \cite{IbukiyamaKatsuraOort}, Ohashi--Onuki--Kudo--Yoshizumi--Nuida \cite{OhashiEtAl}, and Kudo--Harashita \cite{KudoHarashita} for genus \(2\), \(3\), and \(4\), respectively.

\medskip
\noindent\textbf{A General Question.} How large can \(\mu(q)\) be?

\section*{Declaration of competing interest}

The authors declare that they have no known competing financial interests or personal relationships that could have appeared to influence the work reported in this paper.

\section*{Data availability statement}No data are associated with this paper.

\end{document}